\documentclass[11pt,twoside,reqno]{amsart}
\usepackage[dvips,bookmarks=false]{hyperref}
\usepackage{amsmath,amssymb,verbatim}
\usepackage[all]{xy}
\usepackage{graphicx}
\usepackage{geometry}
\date{\today}

\def\deg{\text{deg}\,}

\def\R{{\mathbf R}}
\def\C{{\mathbf C}}
\def\P{{\mathbf P}}

\def\O{{\mathcal O}}

\def\Q{{\mathbf Q}}
\def\Re{{\rm Re\,  }}

\def\1{\mathbf 1}

\def\Z{{\mathbf Z}}
\def\N{{\mathbf N}}

\def\Atrans{\check A}
\def\A{A}

\def\be{\begin{equation}}
\def\ee{\end{equation}}

\DeclareMathOperator{\id}{id}

\DeclareMathOperator{\PL}{{\mathrm PL}}
\DeclareMathOperator{\Vol} {Vol}
\DeclareMathOperator{\conv}{conv}

\newtheorem{thm}{Theorem}[section]
\newtheorem{lma}[thm]{Lemma}

\newtheorem{prop}[thm]{Proposition}

\newtheorem*{thmA}{Theorem A}
\newtheorem*{thmA'}{Theorem A'}
\newtheorem*{thmB'}{Theorem B'}
\newtheorem*{thmB}{Theorem B}
\newtheorem*{thmC}{Theorem C}
\newtheorem*{thmC'}{Theorem C'}
\newtheorem*{corD}{Corollary D}
\newtheorem*{corD'}{Corollary D'}
\theoremstyle{definition}

\theoremstyle{remark}

\newtheorem{preremark}[thm]{Remark}
\newtheorem{preex}[thm]{Example}

\newenvironment{remark}{\begin{preremark}}{\qed\end{preremark}}
\newenvironment{ex}{\begin{preex}}{\qed\end{preex}}

\numberwithin{equation}{section}

\begin{document}

\title[Stabilization of monomial maps in higher
  codimension]{Stabilization of monomial maps in \\ higher codimension}

\date{\today}
\thanks{}
\author{Jan-Li Lin}
\address{Department of Mathematics, Indiana University, Bloomington,
  IN 47405, USA}
\email{janlin@indiana.edu}

\author{Elizabeth Wulcan}
\address{Department of Mathematics, Chalmers University of Technology and the University of Gothenburg\\ SE-412 96 G{\"o}teborg, Sweden}
\email{wulcan@chalmers.se}



\begin{abstract}
A monomial self-map $f$ on a complex toric variety is said to be $k$-stable if
the action induced on
the $2k$-cohomology is compatible with iteration.
We show that under suitable conditions
on the eigenvalues
of the matrix of exponents of $f$,
we can find a toric model with at
worst quotient singularities where
$f$ is $k$-stable.
If $f$ is replaced by an iterate one can find a $k$-stable model
as soon as
the dynamical degrees
$\lambda_k$ of $f$ satisfy $\lambda_k^2>\lambda_{k-1}\lambda_{k+1}$.
On the other hand, we give examples of monomial maps $f$, where this condition is
not satisfied and where the degree sequences $\deg_k(f^n)$ do not
satisfy any linear recurrence.
It follows
that such an $f$ is not $k$-stable on any toric model with at worst quotient singularities.

\end{abstract}

\maketitle
\section*{Introduction}
When studying the dynamics of a dominant meromorphic self-map
$f:X\dashrightarrow X$ on a compact complex manifold $X$ it is often desirable that the action of $f$ on
the cohomology of $X$ be compatible with iterations. Following
Sibony~\cite{S} and Dinh-Sibony \cite{DS09} (see also~\cite{FS}) we
then say that $f$ is \emph{(algebraically) stable}. More precisely, if $f^*$ denotes the
induced action on $H^{2k}(X)$ we say that $f$ is \emph{$k$-stable} if
$(f^n)^*=(f^*)^n$ for all $n$. For examples of classes of $k$-stable
maps see, e.g., \cite{DS09, DS05b}.

If $f$ is not stable, one can look for a model $X'$, birational to $X$,
where the induced self-map on $X'$ is stable.
As shown by Favre \cite{Fa} this is not always possible to achieve. However, for large classes of
surface maps and monomial maps, one can find models $X'\to X$ (with at
 worst quotient
 singularities), so that $f$
lifts to a $1$-stable map, see \cite{DF, Fa, FJ, L1, JW,
   L3}.

In this paper we address
 the question of finding a $k$-stable model for the special class of
 monomial maps, but for arbitrary $k$.
Monomial maps on complex
 projective space $\P^m$, or more generally, on toric varieties, correspond to integer-valued $m\times
 m$-matrices, $M_m(\Z)$.
For $A\in M_m(\Z)$
with entries $a_{ij}$ we write $f_A$
 for the corresponding monomial map
$$f_A(z_1,\ldots, z_m)=(z_1^{a_{11}}\cdots z_m^{a_{m1}}, \ldots,
z_1^{a_{1m}}\cdots z_m^{a_{mm}})$$ with $(z_1,\ldots, z_m)\in
(\C^*)^m$. This mapping is holomorphic on the torus $(\C^*)^m$ and
extends as a rational map to $\P^m$ or to any toric
variety. It is dominant precisely if $\det A\neq 0$. Note that
$f_A^\ell=f_{A^\ell}$.

\begin{thmA}
Assume that the eigenvalues of $A\in M_m(\Z)$ are real and satisfy
 $\mu_1 > \ldots > \mu_m > 0$ or $\mu_1 < \ldots < \mu_m <0$.  Then there is a projective toric variety
$X$, with at worst quotient singularities, such that $f_A:X\dashrightarrow X$ is
$k$-stable for $k=1,\ldots, m-1$.
\end{thmA}
The definition of $k$-stable extends verbatim to toric varities with
at worst quotient singularities, cf. Section ~\ref{maps}.

If the eigenvalues of $A$ only satisfy $|\mu_1|> \ldots > |\mu_m|>0$,
 it is not always possible to find a stable model, see,
e.g., \cite[Example~6.3]{JW}. Still, since
$A^2$ has positive and distinct eigenvalues, by Theorem~A we can find a model so
that $f_A^2$ becomes stable. In fact, there is an
$\ell_0$ such that $f^\ell$ is $k$-stable for $\ell\geq \ell_0$ and
each $k$.


\begin{thmB}
Assume that the eigenvalues of $A\in M_m(\Z)$, ordered so that $|\mu_1|\geq
\ldots \geq |\mu_m|$, satisfy $|\mu_{k_j}|>|\mu_{k_j+1}|$ for
$j=1,\ldots, s$ and $|\mu_m|>0$.
Then there is a projective toric variety
$X$, with at worst quotient singularities, and
$\ell_0\in \N$, such that $f_A^\ell:X\dashrightarrow X$ is
$k_j$-stable for $\ell\geq \ell_0$ and $j=1,\ldots, s$.
\end{thmB}

Recall that the \emph{$k$th degree} $\deg_k(f)$
of the rational self-map $f:\P^m\dashrightarrow
\P^m$ is defined as $\deg f^{-1}(L_k)$ where $L_k$ is a generic linear
subspace of $\P^m$ of codimension $k$.
In \cite{FW, L2} it was proved that the \emph{$k$th dynamical degree}
$$\lambda_k=\lambda_k(f_A):=\lim_n(\deg_k (f_A^n))^{1/n}$$ of $f_A,$
introduced by Russakovskii-Shiffman \cite{RS},
is equal to $|\mu_1|\cdots |\mu_k|$,
 if the eigenvalues of $A$ are ordered so
that $|\mu_1|\geq \ldots \geq |\mu_m|$.
It follows that the condition
$|\mu_k|>|\mu_{k+1}|$ is equivalent to
$\lambda_k^2>\lambda_{k-1}\lambda_{k+1}$. In general the dynamical
degrees satisfy $\lambda_k^2\geq \lambda_{k-1}\lambda_{k+1}$; for
this and other basic properties of
dynamical degrees, see, e.g., \cite{DS05a, G, RS}.
Thus, in particular, Theorem~B says that if we are only interested in the action of $f_A^*$ on
$H^{2k}(X)$,  we can find good models as soon as
$\lambda_k^2>\lambda_{k-1}\lambda_{k+1}$.
One could ask if this is true for general meromorphic $f:X\dashrightarrow X$. Is
it always possible to find a model $X'$ birational to $X$ so that
$f^\ell$ is $k$-stable for $\ell$ large enough when
$\lambda_k^2>\lambda_{k-1}\lambda_{k+1}$?
The problem of finding stable models for $f$ is related to the question
whether the degree sequence $\deg_k(f^n)$ satisfies a linear
recurrence.

\begin{thmC}
For $1\leq k\leq m-1$, assume that the eigenvalues of $A\in M_m(\Z)$
satisfy
\begin{equation}\label{film}
|\mu_{k-1}|>|\mu_k|=|\mu_{k+1}|>|\mu_{k+2}|
\end{equation}
and moreover that $\mu_k/\mu_{k+1}$ is not a root of unity.
Then the degree sequence $\deg_{k}(f_A^n)$ does not satisfy
any linear recurrence.
\end{thmC}

If $k=1$ or $k=m-1$ the condition \eqref{film} on the moduli of the $\mu_j$ should
be interpreted as $|\mu_1|=|\mu_2|>|\mu_3|$ and
$|\mu_{m-2}|>|\mu_{m-1}|=|\mu_{m}|$, respectively.


\begin{corD}
Assume that $A\in M_m(\Z)$ satisfies the assumption of Theorem~C.
Then for each toric projective variety $X$ with at worst quotient
singularities,
$f_A:X\dashrightarrow X$ is not $k$-stable.
\end{corD}

In fact, Corollary ~D follows from a slight generalization of Theorem~C, Theorem ~C', which asserts
that $\deg_{D,k}(f_A^n)$ does not satisfy any linear recurrence, where
$\deg_{D,k}(f_A)$ is the $k$th degree of $f_A$ on a projective toric
variety $X$ with respect to an ample divisor $D$ on $X$, see Section \ref{degreesection}.

Note that if $A$ satisfies the assumption of Theorem ~C, then all powers
$A^\ell$ of $A$ satisfy the assumption as well. Thus we get
that for each $X$ as in Corollary ~D and each
$\ell\in\N$, $f_A^\ell:X\dashrightarrow X$ is not
$k$-stable.
%
%
It would be interesting to investigate whether one can remove the
conditions $|\mu_{k-1}|>|\mu_k|$ and $|\mu_{k+1}|>|\mu_{k+2}|$.
Is it true that $f_A$ cannot
be made
$k$-stable as soon as $|\mu_k|=|\mu_{k+1}|$ and $\mu_k/\mu_{k+1}$ is
not a root of unity?



\smallskip

For $m=2$, Theorems~A and
~B follow from \cite{Fa} and for $m=3$ they follow from
\cite[Theorem~1.1]{L3}.
Moreover, for $k=1$
Theorem~A follows from \cite[Theorem~A]{JW}.
In fact, if $f_A$ is a monomial map on a toric variety $X$, under the assumption $\mu_1> \ldots >\mu_m>0$ one
can find a birational modification $\pi: X'\to X$, with at worst quotient singularities, such that the
lifted mapping $\pi^{-1}\circ f_A \circ \pi: X'\dashrightarrow X'$ is $1$-stable.

Geometrically, $f:X\dashrightarrow X$ is $1$-stable if no iterate of
$f$ sends a hypersurface into the indeterminacy set of $f$, see
\cite{FS, S}. If $f=f_A$ is monomial and $X$ is toric this translates into a
certain condition in terms of the action of $A$ on the fan of $X$, see
\cite[Section~4]{L1} and
\cite[Section~2.4]{JW}.
The construction of a $1$-stable model $X'\to X$ amounts to carefully
refining the fan corresponding to $X$. A model $X'$ that is only birationally equivalent to $X$ can
be obtained in a much less technical way and also for a larger class of
monomial mappings;
for $s=1$ and $k_1=1$, Theorem~B appeared in \cite[Theorem~4.7]{L1} and
\cite[Theorem~B']{JW}, cf.\ Remark ~\ref{compare}.

For $k\geq 2$ we do not in general understand what it means
geometrically to be
$k$-stable, nor if there is a translation into the language of fans
for monomial maps.
In this paper we consider a certain class of toric varieties, where the action of
$f_A^*$ is particularly simple.
Given a basis $\epsilon_1,\ldots, \epsilon_m$ of $\Q^m$
we construct a toric variety $X$, see Section ~\ref{skew}, for which
the
entries of the matrix of $f_A^*:H^{2k}(X)\to H^{2k}(X)$
are the
absolute values of the $k\times k$-minors of $A$ in the basis
$\epsilon_j$ (modulo multiplication by a positive constant). It
follows that a sufficient
condition for $f_A$ to be $k$-stable is that all $k\times k$-minors have the same
sign, see Lemma ~\ref{stable} and Remark~\ref{bakat}.
The basic idea of the proofs of Theorems ~A and ~B is to find bases $\epsilon_j$ so that
this condition is satisfied. The construction will be based on
(strictly) totally positive
matrices, i.e., matrices whose minors are all (strictly)
positive. Typical examples of totally positive matrices are
certain Vandermonde matrices.

Corollary D is due to Favre \cite{Fa} for $m=2$; he showed that if
$|\mu_1|=|\mu_2|$ and $\mu_1/\mu_2$ is
not a root of unity
there
is no model such that (any power of) $f_A$ is stable.
Bedford-Kim \cite{BK} proved Theorem~C for $k=1$ and some cases when
$k>1$, see also
\cite[Theorem~4.7]{L1}.
Following ideas due to Hasselblatt-Propp
\cite{HP} and Bedford-Kim \cite{BK},
we prove Theorem ~C by comparing the degree sequence $\deg_k(f_A^n)$ to a
certain other sequence $\beta_n$, which satisfies a linear
recurrence.
If $\deg_k(f_A^n)$ satisfied a linear recurrence the
set of $n$ for which $\deg_k(f_A^n)=\beta_n$ would be eventually
periodic, which we show cannot be the
case, see. To do this we express $\deg_k(f_A^n)$ in terms of minors of
$A^n$ using a result from \cite{FW}, which expresses $\deg_{k}(f_A)$
as a mixed volume of certain polytopes, and
a method due to Huber-Sturmfels \cite{HS} of computing mixed volumes of
polytopes.




\smallskip

\textbf{Acknowledgment}
We thank Eric Bedford, Mattias Jonsson, and Pavlo Pylyavskyy for fruitful
discussions.
We also thank the referee for many helpful suggestions.
Part of this work was done while the authors were visiting the
Institute for Computational and Experimental Research in Mathematics
(ICERM); we would like to thank the ICERM for their hospitality.
The second author was supported by the Swedish Research Council.

\section{Toric varieties}\label{sec:toric}

A complex toric variety is a (partial) compactification of the torus $T\cong (\C^*)^m$, which contains $T$ as a dense subset and which admits an action of $T$ that extends the natural action of $T$ on itself. We briefly recall some of the basic definitions, referring to~\cite{Fu} and~\cite{O} for details.

\subsection{Fans and toric varieties}\label{fansoch}

Let $N$ be a lattice isomorphic to $\Z^m$ and let $M=\text{Hom}(N,\Z)$
denote the dual lattice. Set $N_\Q:=N\otimes_\Z \Q$,
$N_\R:=N\otimes_\Z \R$, and define $M_\Q$ and $M_\R$ analogously. Let $\R_+$ and $\R_-$ denote the sets of
non-negative and non-positive numbers, respectively.

A \emph{cone} $\sigma$ in $N_\R$ is a set that is closed under positive scaling. If $\sigma$ is convex and does not contain any line in $N_\R$, it is said to be \emph{strictly convex}. If $\sigma$ is of the form $\sigma=\sum\R_+v_i$ for some $v_i\in N$,
we say that $\sigma$ is a \emph{convex rational cone} \emph{generated}
by the vectors $v_i$. A \emph{face} of $\sigma$ is the intersection of
$\sigma$ and a \emph{supporting hyperplane},
i.e., a hyperplane through the origin such that the whole cone
$\sigma$ is contained
in one of the closed half-spaces determined by the hyperplane.
The \emph{dimension} of $\sigma$
is the dimension of the linear space $\R\sigma$ spanned by
$\sigma$. One-dimensional faces of $\sigma$ are called \emph{edges}
and one-dimensional cones are called \emph{rays}. Given a ray
$\sigma$, the associated \emph{primitive vector} is the first lattice
point met along $\sigma$. The \emph{multiplicity}
$\text{mult}(\sigma)$ of $\sigma$ is the index of the lattice generated by the
primitive elements of the edges of $\sigma$ in the lattice generated
by $\sigma$.
A $k$-dimensional cone is \emph{simplicial}
if it can be generated by $k$ vectors.
A cone is \emph{regular} if
it is simplicial and of multiplicity one.

A \emph{fan} $\Delta$ in $N$ is a finite collection of rational
strongly convex cones in $N_\R$ such that each face of a cone in
$\Delta$ is also a cone in $\Delta$ and, moreover, the intersection of
two cones in $\Delta$ is a face of both of them. Let $\Delta_k$ denote
the set of cones in $\Delta$ of dimension $k$.
The fan $\Delta$ is
said to be \emph{complete} if  the union of all cones in $\Delta$
equals $N_\R$.
If all cones in $\Delta$ are simplicial then $\Delta$ is said to be
\emph{simplicial}, and if all cones are regular, $\Delta$ is said to
be \emph{regular}.
A fan $\widetilde\Delta$ is a \emph{refinement} of $\Delta$ if each cone in $\Delta$ is a union of cones in $\widetilde\Delta$.

A fan $\Delta$ determines a toric variety $X(\Delta)$ obtained by patching together affine
toric varieties $U_\sigma$ corresponding to the cones
$\sigma\in\Delta$.
It is compact if and only if $\Delta$ is
complete. Toric varieties are normal and Cohen-Macaulay.
The variety $X(\Delta)$ is nonsingular if and only if $\Delta$
is regular. $X(\Delta)$ has at worst quotient singularities, i.e.,
it is locally the quotient of a smooth variety by the action of a
finite group, if and only if
$\Delta$ is simplicial, see e.g. 
\cite[Section 2.2]{Fu}.
In this case, we will also say that the variety $X(\Delta)$ is {\em simplicial}.
For any fan $\Delta$ in $N$ there is a fan $\widetilde \Delta$ such
that $X(\widetilde \Delta)\to X(\Delta)$ is a resolution of
singularities.

\subsection{Cohomology of toric varieties and piecewise linear
  functions}\label{coho}

Let $\Delta$ be a simplicial complete fan. Then the odd cohomology
groups 
of $X:=X(\Delta)$ vanish and the even
cohomology groups are generated by varieties invariant under the action of $T$. More precisely $H^{2k}(X):=H^{2k}(X; \R)$ is
generated by $T$-invariant varieties of codimension $k$.
There is a Hodge
decomposition $H^k(X)\otimes_\R \C=\bigoplus_{p+q=k}H^{p,q}(X)$ of the cohomology
groups of $X$ and, moreover, $H^{p,q}(X)=0$ if $p\neq q$, see,
e.g., \cite[Proposition~12.11]{Da} and \cite[Chapter~2.5]{PS}.
In particular,
\[
H^{2k}(X)=H^{k,k}(X;\R):=H^{k,k}(X)\cap H^{2k}(X;\R).
\]

Each cone $\sigma\in\Delta_k$ determines an irreducible subvariety $V(\sigma)$ of $X$
of codimension $k$ that is invariant
under the action of $T$.
If we use $[V(\sigma)]$ to denote the class of $V(\sigma)$ in $H^{2k}(X)$, then
the classes $[V(\sigma)]$, as $\sigma$ runs through all cones of codimension $k$,
generate $H^{2k}(X)$. In particular, each ray $\rho$ in $\Delta$ determines a
$T$-invariant prime Weil divisor $D(\rho)$ and these divisors generate
$H^{2}(X)$. 
Since $\Delta$ is simplicial,
$$\frac{1}{\text{mult}(\sigma)}[V(\sigma)]=\prod [D(\rho_i)]$$
in $H^{*} (X)$,
where $\rho_i$ are the edges of $\sigma$, i.e., the
$[D(\rho_i)]$ genererate $H^{*}(X)$ as an algebra.

Let $\PL(\Delta)$ be the set of all continuous functions
$h:\bigcup_{\sigma\in\Delta} \sigma \to \R$ that are
\emph{piecewise linear with respect to $\Delta$}, i.e., for each cone
$\sigma\in\Delta$ there exists $m=m(\sigma)\in M$ with
$h|_\sigma=m$. A function in $\PL(\Delta)$ is said to be
\emph{strictly convex} if it is convex and defined by different elements $m(\sigma)$ for
different cones $\sigma\in\Delta_m$. A compact toric variety $X(\Delta)$ is
projective if and only if there is a strictly convex $h\in
\PL(\Delta)$. We then say that $\Delta$ is \emph{projective}.

Functions in $\PL(\Delta)$ are in one-to-one correspondence with
$T$-invariant Cartier divisors. If $D$ is a $T$-invariant Cartier divisor of the
form $D=\sum a_i D(\rho_i)$, then the corresponding function  $h_D\in
\PL(\Delta)$ is determined by $h_D(v_i)=a_i$ if $v_i$ is a primitive
vector for $\rho_i$. Conversely $h\in\PL(\Delta)$ determines the Cartier divisor $D(h):=\sum h(v_i) D(\rho_i)$.
Given $h_1, h_2\in\PL(\Delta)$, the corresponding divisors are
linearly equivalent if and only if $h_1-h_2$ is linear.
The function $h_D$ is strictly convex if and only if $D$ is ample.

A function $h\in\PL(\Delta)$ determines a non-empty polyhedron
$$
P(h) := \{ m \in M_\R , \,  m \leq h  \} \subset M_\R~;
$$
in particular,
$$
P_D:=P(h_D)=\{m\in M_\R, \, m (v_i)\leq a_i \}.
$$
If $h$ is convex, then $P(h)$ is a compact \emph{lattice polytope} in $M_\R$, i.e.,  it is the convex hull of finitely many points
in the lattice $M$.
 Conversely, if $P\subset M_\R$ is a lattice polytope, then
the function
\begin{equation}\label{hp}
h_P (u) := \sup \{ m(u), \, m \in P \}
\end{equation}
is a piecewise linear convex function on $N_\R$. If
$h_P\in\PL(\Delta)$ then $\Delta$ is said to be \emph{compatible} with
$P$. We write $D_P$ for the corresponding divisor on $X(\Delta)$.

\subsection{Mixed volume and intersection of divisors}\label{mixed}
Given any finite collection of convex compact sets $K_1,\ldots,
K_s\subset M_\R$, we let $K_1+\cdots + K_s$ denote the \emph{Minkowski
  sum}
$$K_1 + \cdots + K_s:=\{x_1+\cdots +x_s \mid x_j\in K_j\},$$
and
for $r\in\R_+$, we write $rK_j:=\{rx\mid x\in
K_j\}$. Let $\Vol$ be the Lebesque measure on $M_\R\cong \R^m$ normalized so that
the parallelepiped $$Q_e:=\Bigl\{\textstyle{\sum_{j=1}^m a_j e_j\mid 0\leq a_j\leq 1}\Bigr\},$$
spanned by a basis $e_1,\ldots, e_m$ of $M$, has volume 1.

A theorem by Minkowski and Steiner asserts that
$\Vol(r_1K_1+\cdots +r_sK_s)$ is a homogeneous polynomial  of degree $m$ in the variables
$r_1,\ldots, r_s\in \R$. In particular, there is a unique expansion:
\begin{equation}\label{minkowski}
\Vol\left (r_1K_1+\cdots +r_sK_s\right )=
\sum_{k_1+\cdots +k_s=m} \binom{m}{k_1, \ldots, k_s}
\Vol\left (K_1[k_1],\ldots, K_s[k_s]\right )\,  r_1^{k_1}\cdots r_s^{k_s};
\end{equation}
the coefficients $\Vol(K_1[k_1],\ldots, K_s[k_s])\in \R$ are called
\emph{mixed volumes}. Here the notation $K_j[k_j]$ refers to the repetition of $K_j$ $k_j$
times.

\begin{ex}\label{rakning}
Pick $u_1, \ldots, u_m\in M_\R$ and let $P_j$ be the line segments
$[0,u_j]\subset M_\R$. Then $r_1P_1+\cdots +r_mP_m$ is the parallelepiped
$Q_{ru}$, where $ru$ denotes the tuple $r_1u_1,\ldots, r_mu_m$, and so
$$\Vol(r_1P_1+\cdots +r_mP_m)= r_1\cdots r_m\Vol(Q_u).$$ Hence
$\Vol(P_1,\ldots,P_m)=\Vol(Q_u)/m!$. Note that $\Vol(Q_u)$ is strictly
positive if and only if the $u_j$ are linearly independent.
\end{ex}

If $\Delta$ is compatible with $P_1,\ldots, P_s$, then the
intersection product (i.e., the cup product for cohomology classes) of the corresponding divisor classes equals
\begin{equation}\label{snitt}
[D_{P_1}]^{k_1}\cdots[D_{P_s}]^{k_s}= m!\Vol (P_1[k_1],\ldots, P_s[k_s])
\end{equation}
if $k_1+\cdots +k_s=m$,
see \cite[p.~79]{O}.

\section{Monomial maps}\label{maps}
Given a group homomorphism $A: M\to M$, we will write $A$ also for the induced
linear maps $M_\Q\to M_\Q$ and $M_\R\to M_\R$. Moreover, we let $\Atrans$
denote the
dual map $N\to N$, as well as the dual linear maps $N_\Q\to N_\Q$ and
$N_\R\to N_\R$.
It turns out to be convenient to use this notation rather than
writing $\A$ for the map on $N$ and $\Atrans$ for the map on $M$.

Let $\Delta$ be a fan in $N\cong \Z^m$. Then any group
homomorphism $\Atrans : N \to N$ gives rise to a rational map $f_A:
X(\Delta)\dashrightarrow X(\Delta)$, which is equivariant with
respect to the action of $T$.
Let $e_1,\ldots , e_m$ be a basis of $M$ and let $e_1^*,\ldots, e_m^*$
be the dual basis of $N$. Then the dual map $A:M\to M$ is of the form
$A=\sum a_{ij} e_i\otimes e_j^*$ for some $a_{ij}\in\Z$. If
$z_1,\ldots,z_m$ are the induced coordinates on $T$, then $f_A$ is the
monomial map
$$f_A(z_1,\ldots, z_m)=(z_1^{a_{11}}\cdots z_m^{a_{m1}}, \ldots,
z_1^{a_{1m}}\cdots z_m^{a_{mm}})$$ restricted to $T$.
Conversely, any rational, equivariant map $f: X(\Delta)\dashrightarrow
X(\Delta)$ comes from a group homomorphism $N\to N$, see~\cite[p.19]{O}.

The map $f_A: X(\Delta)\dashrightarrow X(\Delta)$ is holomorphic precisely if $\Atrans:N_\R\to N_\R$
satisfies that for each $\sigma\in \Delta$ there is a $\sigma'\in\Delta$, such that $\Atrans(\sigma)\subseteq \sigma'$.
Then
$f_A^*[D(h)]=[D(h\circ \Atrans)]$, see, e.g., \cite[Chapter~6,
  Exercise~8]{M}, and, moreover,
$P(h\circ\Atrans)=A P(h)$.
Given a fan $\Delta$ and a group homomorphism $\Atrans: N\to N$ one can find a
regular refinement $\widetilde
\Delta$ of $\Delta$ such that the induced equivariant map $\tilde f_A:X(\widetilde \Delta)\to
X(\Delta)$ is holomorphic.
We denote by $\pi$ the modification $ X(\widetilde\Delta)\to
X(\Delta)$ induced by the identity map $\id:N\to
N$.
Furthermore, we have the relation $\tilde f_A = f_A \circ \pi$, i.e.,
the following diagram commutes.
\[
\xymatrix{
& X(\widetilde \Delta)\ar[ld]_{\pi}\ar[rd]^{\tilde{f}_A}\\
X(\Delta)\ar@{-->}[rr]_{f_A} & & X(\Delta)
}
\]

%
%
%
%
%
Now the pullback of a $T$-invariant Cartier divisor $D$ under
$f_A: X(\Delta)\dashrightarrow X(\Delta)$ is defined as $f_A^* D= \pi_*\tilde f_A^* D$;
in fact, this definition does not depend on the particular choice of
$\widetilde \Delta$. The divisor $f_A^*D$ is in general only
$\Q$-Cartier, cf. \cite[Chapter~3.3]{Fu}.
Note that, since $H^*(X(\Delta))$ is generated (as an algebra) by Cartier divisors,
$f_A$ induces an action $f_A^*$ on $H^*(X(\Delta))$.

\section{An important example}\label{skew}
We will prove Theorems A and B by constructing toric varieties of a
certain type.
Throughout this paper we let $I=\{i_1,\ldots, i_k\}$ and
$J=\{j_1,\ldots, j_\ell\}$ be strictly increasing multi-indices in
$\{1,\ldots, m\}$.  If $|I|=|J|=k$, we let
$B_{IJ}$ denote the minor corresponding to the sub-matrix of $B$ with
rows $i_1,\ldots, i_k$ and columns $j_1,\ldots, j_k$. Moreover, we write $[\ell]$ for the multi-index
$\{1,\ldots ,\ell\}$ and $I^C$ for the complement $[m]\setminus  I$ of $I$.


Pick linearly independent vectors $v_1, \ldots, v_m\in N_\Q$
and let
$\Delta$ be the fan
\[
\Delta=\Bigl\{\textstyle\sum_{j=1}^m\R_+\varepsilon_jv_j\Bigr\}_
{\varepsilon=(\varepsilon_1,\ldots,\varepsilon_m)\in\{0,-1,+1\}^m}.
\]
In particular, the rays of $\Delta$ are of the form $\R_+ v_j$ and $\R_- v_j$. For
simplicity we will assume that $v_j$ is the primitive vector of the
ray $\R_+v_j$ for each $j$. Note that $\Delta$ is complete and simplicial and that there are strictly
convex functions in $\PL(\Delta)$; hence the resulting toric variety
$X(\Delta)$ is projective and has at worst quotient singularities.
If the
$v_j$ form a basis of $N$, then $X(\Delta)$ is isomorphic to $(\P^1)^m$. In
general we will think of $X=X(\Delta)$ as a ``skew product'' of
$\P^1$s.

Note that the rays of $\Delta$ determine divisors
$D_j:=D(\R_- v_j)$ and $E_j:=D(\R_+ v_j)$, such that $E_j$ is linearly
equivalent to
$D_j$ for each $j$.
The polytope $P_j:=P_{D_j}$ associated to the divisor $D_j$ is the line segment in
$M_\R$ with the origin and $u_j\in M_\R$ as endpoints,
where $u_j$ is the point in $M_\R$ such that $\langle v_i, u_j\rangle= \delta_{ij}$ (Kronecker's delta).
Notice that the $u_j$, as vectors, are linearly independent.
By \cite[Section~5.2]{Fu} $H^{2k}(X)$ will be generated by the intersection (cup) product of divisor classes:
$$[D_I]:=[D_{i_1}]\cdots [D_{i_k}]$$ for $I=\{i_1, \ldots,
i_k\}\subseteq [m]$. 
In particular
\[
f_A^*[D_I]=\sum_{|J|=k}\alpha_{I J} [D_J]
\]
for some
coefficients $\alpha_{I J}$.
Recall that $I^C=[m]\setminus I$, thus
from \eqref{snitt} we get

\begin{equation}
[D_I]\cdot[D_{J^C}]= \left \{
\begin{array}{cl}
m!\Vol(P_1,\ldots, P_m)>0  & \text{if } J=I\\
0  & \text{otherwise }
\end{array} \right. ,
\end{equation}
cf. Example ~\ref{rakning}.
It follows that
\[
f_A^*[D_I]\cdot [D_{J^C}]= \alpha_{IJ} \cdot m!\Vol(P_1,\ldots, P_m)
\]
On the other hand, for $I=\{i_1, \ldots, i_k\}$ and $J^C=\{j_1,\ldots, j_{m-k}\}$,
by the projection formula \cite[p.325]{FuInt}, we have
\begin{multline*}\label{lala}
f_A^*[D_I]\cdot [D_{J^C}] = \pi_* \tilde f_A^*[D_I] \cdot [D_{J^C}] =
\tilde f_A^*[D_I]\cdot \pi^*[D_{J^C}] = \\
\tilde f_A^*\big ([D_{i_1}]\cdots [D_{i_k}] \big )\cdot \pi^*\big ([D_{j_1}]\cdots
[D_{j_{m-k}}]\big ) = \\ \tilde f_A^*[D_{i_1}]\cdots \tilde f_A^*[D_{i_k}]\cdot
\pi^*[D_{j_1}]\cdots \pi^*[D_{j_{m-k}}],
\end{multline*}
where the last step follows since $\tilde f_A$ and $\pi$ are
holomorphic.
Recall from Section
~\ref{maps} that the polytopes associated to $\tilde f_A^*[D_i]$ and
$\pi^*[D_j]$ are $AP_i$ and $\id P_j=P_j$, respectively. Thus in light
of \eqref{snitt},
\begin{multline*}
\tilde f_A^*[D_{i_1}]\cdots \tilde f_A^*[D_{i_k}]\cdot
\pi^*[D_{j_1}]\cdots \pi^*[D_{j_{m-k}}] =
m! \Vol(AP_{i_1}, \ldots, AP_{i_k}, P_{j_1}, \ldots, P_{j_{m-k}}).
\end{multline*}
%
Let $A_{IJ}=A_{IJ}(u_j)$ denote the minors of $A:M_\R\to M_\R$
with respect to the basis $u_1,\ldots, u_m$. Then, in light of Example
~\ref{rakning},
\[
\Vol(AP_{i_1}, \ldots, AP_{i_k}, P_{j_1}, \ldots,
P_{j_{m-k}})= | A_{IJ}| \Vol(P_{1}, \ldots, P_{m}).
\]

To conclude,  $\alpha_{IJ}=|A_{IJ}|$, and thus we have proved the following result.

\begin{lma}\label{pullback}
Let $\Delta$ be a fan of the form $\Delta=\{\sum_{j=1}^m\R_+\varepsilon_jv_j\}_{\varepsilon\in\{0,-1,+1\}^m}.$ Using
  the notation above,
\begin{equation*}
f_A^*[D_I]=\sum_{|J|=k}| A_{IJ}|[D_J].
\end{equation*}
\end{lma}

Hence
\[
(f_A^*)^\ell[D_I]=\sum_{|J_1|=\ldots= |J_{\ell-1}|=|J|=k} | A_{IJ_1}|| A_{J_1J_2}|\cdots| A_{J_{\ell-2}J_{\ell-1}}|| A_{J_{\ell-1}J}|[D_J]
\]
and
\[
(f_A^\ell)^*[D_I]=(f_{A^\ell})^*[D_I]=\sum_{|J|=k} | A^\ell_{IJ}|[D_J],
\]
where $ A^\ell_{IJ}$ denotes the $IJ$-minor of $A^\ell$. Recall the
Cauchy-Binet formula:
\begin{equation}\label{CB}
(AB)_{IJ}=\sum_{|K|=k}  A_{IK} B_{KJ}.
\end{equation}
It follows that a sufficient condition for $(f_A^*)^\ell=(f_A^\ell)^*$ is
that $ A_{IJ}\geq 0$ for all $I,J$ or $ A_{IJ} \leq 0$ for
all $I,J$. Let us summarize this:

\begin{lma}\label{stable}
Let $\Delta$ be a fan of the form
$\Delta=\{\sum_{j=1}^m\R_+\varepsilon_jv_j\}_{\varepsilon\in\{0,-1,+1\}^m}$.
Using the notation above, if $ A_{IJ} \geq 0$ for all strictly
increasing multi-indices
$I=\{i_1,\ldots, i_k\}, J=\{j_1,\ldots, j_k\}\subseteq [m]$ or if  $ A_{IJ}\leq 0$ for all
$I, J$, then $f_A:X(\Delta)\dashrightarrow X(\Delta)$ is $k$-stable.
\end{lma}

\begin{remark}\label{bakat}
One can also construct the fan $\Delta$ above starting from a basis
$\epsilon_1,\ldots, \epsilon_m$ of
$M_\Q$.
For $j=1,\ldots, m$, let $V_j$
be the one-dimensional subspace of $N_\Q$ defined by
$$V_j=\{v\in N_\Q\mid \epsilon_\ell(v)=0 \text{ for } \ell \neq j\}.$$
Then each $V_j$ determines two rational rays in $N_\R$, which will be the
rays of $\Delta$; more precisely, pick $v_j$ to be a primitive vector
of one of the rays in $V_j$ and construct $\Delta$ as above.
Now the polytopes $P_{D_j}$ and
$P_{E_j}$ will lie in the one-dimensional vector space spanned by
$\epsilon_j$. By the choice of $v_j$ we can arrange so that
$u_j$ is a positive multiple of $\epsilon_j$.
Then the minor $A_{IJ}$ of $A$ in the basis $u_j$ is just
a positive constant times the $IJ$-minor $A_{IJ}(\epsilon_j)$ of $A$
in the basis $\epsilon_j$.
More precisely, if
$\epsilon_j=\alpha_j u_j$, then $$A_{IJ}=\frac{\alpha_{i_1}\cdots
  \alpha_{i_k}}{\alpha_{j_1}\cdots \alpha_{j_k}} A_{IJ}(\epsilon_j).$$
\end{remark}

\begin{ex}\label{projective}
Consider the monomial map
\[
f_A(z_1,\ldots, z_m)=
(z_1^{a_{11}}\cdots z_m^{a_{m1}}, \ldots, z_1^{a_{1m}}\cdots
z_m^{a_{mm}}).
\]
If all $k\times k$-minors of the matrix $(a_{ij})$ are
either nonnegative (or nonpositive), then $f_A:(\P^1)^n\dashrightarrow
(\P^1)^n$ is $k$-stable. In particular, if $(a_{ij})$ is totally positive (or
totally negative) $f_A$ is stable on $(\P^1)^n$ for all $k$.
Indeed, since $(a_{ij})$ is the matrix of the group homomorphism $A:M\to
M$ associated with $f_A$
with respect to the basis $e_j$ of $M$, cf.\ Section ~\ref{maps},  Lemma ~\ref{stable}
implies that $f_A$ is stable on
$$X \Big (\Big \{\sum_{j=1}^m\R_+\varepsilon_je_j^* \Big
\}_{\varepsilon\in\{0,-1,+1\}^m}\Big
) = (\P^1)^n.$$

\end{ex}

\section{Proof of Theorem~A}\label{proofA}

Given a basis $\xi_1, \ldots, \xi_m$ of $M_\R$, and a linear transformation
$A$,
let $A(\xi_j)$ denote the
matrix of $A$ with respect to this basis.

Assume that $A$ has distinct positive
eigenvalues $\mu_1> \ldots > \mu_m>0$. Then so has the matrix $A(\xi_j)$,
for any basis $\xi_1,\ldots,\xi_m$ of $M_\R$. By \cite{BJ},
one can find a strictly totally positive matrix $A^+$ with eigenvalues
$\mu_1,\ldots,\mu_m$. Since both matrices $A(\xi_j)$ and $A^+$
are diagonalizable over $\R$ and
they have the same set of eigenvalues, it follows that they are
conjugate to each other over $\R$.
Thus, without loss of generality, we can perform a change of basis and assume that
$A(\xi_j)=A^+$.

%
%
The coefficients and the minors of the matrix $A(\xi_j)$
change continuously as one perturbs the basis $\xi_j$. Moreover, being strictly totally positive is an open condition on the
space of matrices.
Hence, by perturbing $\xi_j$, we can find a basis
$\epsilon_1,\ldots, \epsilon_m$ of $M_\Q$ such that $A(\epsilon_j)$ is strictly totally
positive.
Given this basis, following Remark~\ref{bakat}, we construct a fan of the form
\[
\Delta=\Bigl\{\textstyle\sum_{j=1}^m\R_+\varepsilon_j v_j\Bigr\}_{\varepsilon\in\{0,-1,+1\}^m}.
\]
In view of Remark \ref{bakat}, using the notation of
Section~\ref{skew}, all
$k\times k$-minors $A_{IJ}$ in the basis $u_j$ are then positive for $k=1,\ldots,
m-1$, and thus Lemma
~\ref{stable} asserts that $f_A:X(\Delta)\dashrightarrow X(\Delta)$ is $k$-stable for $k=1,\ldots, m-1$.

If $A$ has negative and distinct eigenvalues, by arguments as above,
we can find a basis of $M_\Q$ so that the matrix of $A$ is of the form
$-B$, where $B$ is (strictly) totally positive. Constructing $\Delta$
as above, the $k\times k$-minors of $A$ in the basis $u_j$ will then all have
sign $(-1)^k$ and so, by Lemma ~\ref{stable},
$f_A:X(\Delta)\dashrightarrow X(\Delta)$ is $k$-stable for
$k=1,\ldots, m-1$.

\section{Proof of Theorem~B}\label{proofB}
Given vectors $w_1,\ldots, w_m\in M_\R$ we will write
$w_I=w_{i_1}\wedge\cdots \wedge w_{i_k}$ for $I=\{i_1,\ldots,
i_k\}\subseteq [m]$. 
Note that if $A:M_\R\to M_\R$ is a linear map with eigenvalues $\mu_1,\ldots,
\mu_m$, then the induced linear map $\Lambda^k A: \Lambda^k
M_\R\to \Lambda^k M_\R$ has eigenvalues $\mu_I:=\mu_{i_1}\cdots
\mu_{i_k}$ for $I=\{i_1,\ldots, i_k\}\subseteq [m]$. Throughout we will
assume that the
eigenvalues of $A$ are ordered so that $|\mu_1|\geq \ldots \geq |\mu_m|$.

\begin{lma}\label{orthant}
Given a basis $\rho_1,\ldots,\rho_m$ of $M_\R$, there is a basis
$\epsilon_1,\ldots, \epsilon_m$ of $M_\Q$, such that for $k=1,\ldots, m$, $\rho_{[k]}$
lies in the interior of the first orthant $\sigma_k := \sum_{|I|=k}\R_+ \epsilon_I \subset
\Lambda^k M_\R$, and, moreover, the hyperplane $H_k\subset\Lambda^k M_\R$, spanned by $\rho_I$, $I\neq [k]$, intersects $\sigma_k$ only at the origin.
\end{lma}

\begin{proof}
Pick real numbers $\mu_1>\ldots >\mu_m>0$ and let $A:M_\R\to M_\R$ be a linear map given by a diagonal matrix
in the basis $\rho_j$ with diagonal entries $\mu_1, \ldots, \mu_m$. As in the proof of Theorem~A we can then choose a
basis $\epsilon_1,\ldots, \epsilon_m$ of $M_\Q$ such that $A(\epsilon_j)$ is strictly
totally positive. In particular, for a given $k$, $ A_{IJ}(\epsilon_j)>0$ for all
$I,J$ such that $|I|=|J|=k$, which means that $\Lambda^kA:\Lambda^k M_\R\to \Lambda^k M_\R$ maps the
first orthant $\sigma_k$ into itself. Since the  $\rho_j$ are the
eigenvectors of $A$, it follows by the
Perron-Frobenius theorem
that the eigenvector $\rho_{[k]}$ (or $-\rho_{[k]}$) corresponding to the largest
eigenvalue $\mu_{[k]}$ of $\Lambda^kA$ is contained in
the interior of $\sigma_k$ 
and, moreover, that $H_k\cap\sigma_k$ is the origin.
\end{proof}

To prove Theorem~B, we choose a basis $\rho_j$ such that the linear map $A:M_\R\to M_\R$ is in real Jordan form, i.e., with blocks
\[
\begin{bmatrix}
\mu_j & 1 & &  \\
 & \mu_j & \ddots  & \\
& & \ddots & 1\\
& & & \mu_j
\end{bmatrix}
\text{ and }
\begin{bmatrix}
C_j & I & &  \\
 & C_j & \ddots  & \\
& & \ddots & I\\
& & & C_j
\end{bmatrix},
\text{ where }
C_j=\begin{bmatrix}
a_j & b_j\\
-b_j & a_j
\end{bmatrix}
\]
and $I$ is the $2\times 2$ identity matrix; we have
the first block type for real eigenvalues $\mu_j$ and the second type for
complex eigenvalues $a_j \pm i b_j$. We order the blocks so that
moduli
of the eigenvalues are in decreasing order along the diagonal and
the vectors $\rho_j$ so that $\rho_1$ is an eigenvector corresponding to the largest eigenvalue etc.
Next, we let $\epsilon_1,\ldots, \epsilon_m$ be a basis of $M_\Q$
constructed as in Lemma~\ref{orthant}, and from $\epsilon_j$,
following Remark~\ref{bakat}, we construct a fan $\Delta$ of the
form
\[
\Delta=\Bigl\{\textstyle\sum_{j=1}^m\R_+\varepsilon_j v_j\Bigr\}_{\varepsilon\in\{0,-1,+1\}^m}.
\]

\smallskip

Assume that $|\mu_k|>|\mu_{k+1}|$. Then
$\mu_{[k]}$ is the unique eigenvalue of $\Lambda^k A: \Lambda^k
M_\R\to \Lambda^k M_\R$ of largest
modulus. Since $A$, and thus $\Lambda^kA$, is real, $\mu_{[k]}$ is
real. 
Moreover, $\Lambda^k A \rho_{[k]}= \mu_{[k]} \rho_{[k]}$,
so that $\rho_{[k]}$ spans the one-dimensional eigenspace of
$\mu_{[k]}$. By Lemma ~\ref{orthant} $\rho_{[k]}$ is the unique (up to scaling) eigenvector of
$\Lambda^k A$ that is contained in $\sigma_k$ and the hyperplane in
$\Lambda^k M_\R$ spanned by the other eigenvectors intersects
$\sigma_k$ only at the origin, and thus, since $\mu_{[k]}$ is
the unique eigenvalue of largest modulus, $\sigma_k$ will
get attracted to the eigenspace $\R\rho_{[k]}\subseteq \Lambda^k
M_\R$.
Hence there is an $\ell_k\in\N$, such
that $(\Lambda^k A)^\ell\sigma_k\subset \sigma_k$ or
$$(\Lambda^k A)^\ell\sigma_k
\subset -\sigma_k:=\{x\in M_\R\mid -x\in \sigma_k\}$$ for all $\ell\geq
\ell_k$. In particular, for such an $\ell$, $(\Lambda^k A)^\ell
\epsilon_I \in \sigma_k$ for all $I=\{i_1,\ldots, i_k\}\subseteq [m]$ or $(\Lambda^k A)^\ell \epsilon_I \in -\sigma_k$ for $(\Lambda^k A)^\ell$
all $I$.
This means that the entries of $(\Lambda^k A)^\ell$, i.e., the $k\times k$-minors of $A^\ell$, in the basis $\epsilon_j$ are
either all positive or all negative. In view of Remark~\ref{bakat}, using the notation of
Section~\ref{skew},
this implies that
$A^\ell_{IJ}(u_j)\geq 0$ for all $I=\{i_1, \ldots, i_k\}, J=\{j_1, \ldots,
j_k\}\subseteq [m]$ or $ A^\ell_{IJ}(u_j)\leq 0$ for all $I, J$. Now Lemma
~\ref{stable} asserts that $f_A^\ell:X(\Delta)\dashrightarrow X(\Delta)$ is $k$-stable.

Finally let $\ell_0=\max_j \ell_{k_j}$. Then
$f_A^\ell:X(\Delta)\dashrightarrow X(\Delta)$ is $k_j$-stable for
$\ell\geq \ell_0$ and $j=1,\ldots, s$.

\begin{remark}\label{compare}
For $s=1$ and $k_1=1$ Theorem~B appeared as Theorem~4.7 in
\cite{L1} and Theorem~B' in \cite{JW}. 
Also in these papers the idea of the proof
is to find a basis (of $N_\R$) such that the first orthant is mapped
into itself and then construct a toric variety as in
Section~\ref{skew}.

\end{remark}

\section{Degrees of monomial maps}\label{degreesection}


Let $\Delta$ be a complete simplicial projective fan and let $D$ be an
ample divisor on $X(\Delta)$. Then the \emph{$k$th degree of $f_A:X(\Delta)\dashrightarrow X(\Delta)$ with respect to $D$} is defined as
\[
\deg_{D,k}:=f_A^* D^k\cdot D^{m-k}.
\]
If $X(\Delta)=\P^m$ and $\mathcal O(D)=\mathcal O_{\P^m}(1)$, then
$\deg_{D,k}$ coincides with the $k$th degree $\deg_k$ as defined
in the introduction.

We have the following more general version of Theorem C. Indeed, Theorem~C corresponds to the case when $X(\Delta)=\P^m$ and $\mathcal
O(D)=\mathcal O_{\P^m}(1)$.

\begin{thmC'}
Let $\Delta$ be a complete simplicial fan and let $D$ be an ample
divisor on $X(\Delta)$.
For $1\leq k\leq m-1$ assume that the eigenvalues of $A\in M_m(\Z)$
satisfy 
\begin{equation}\label{ballong}
|\mu_{k-1}|>|\mu_k|=|\mu_{k+1}|>|\mu_{k+2}|
\end{equation}
and moreover that $\mu_k/\mu_{k+1}$ is not a root of unity.
Then the degree sequence $\deg_{D,k}(f_A^n)$ does not satisfy
any linear recurrence.
\end{thmC'}


If $k=1$ or $k=m-1$ the condition \eqref{ballong} on the moduli  of the $\mu_j$ should
be interpreted as $|\mu_1|=|\mu_2|>|\mu_3|$ and
$|\mu_{m-2}|>|\mu_{m-1}|=|\mu_{m}|$, respectively.

\begin{remark}
Note that there exist maps that satisfy the assumption of Theorem
~C'.
For example, choose integers $a_1,\ldots, a_{k-1},
b_1,b_2,a_{k+2},\ldots, a_m$ such that
\[
|a_1|\geq \ldots \geq
|a_{k-1}|>\sqrt{b_1^2+b_2^2} > |a_{k+2}|\geq \ldots \geq |a_m|
\]
 and
$b_1+ib_2=\sqrt{b_1^2+b_2^2} \cdot e^{2\pi i \theta}$, where $\theta\notin\Q$. Then (the matrix of) the monomial map
\[
(z_1,\ldots, z_m)\mapsto (z_1^{a_1},\ldots, z_{k-1}^{a_{k-1}},
z_k^{b_1}z_{k+1}^{b_2}, z_k^{-b_2}z_{k+1}^{b_1}, z_{k+2}^{a_{k+2}},
\ldots, z_m^{a_m})
\]
satisfies the assumption of Theorem ~C'.
\end{remark}

Corollary D follows immediately from Theorem C' and
the following result. This is probably well-known, but we include a
proof for completeness; we follow the proof of Proposition ~3.4 in
\cite{L3}.



%

\begin{prop}\label{stabdeg}
Assume that $\Delta$ is a simplicial projective fan and let $D$ be an ample divisor on
$X(\Delta)$. Assume that the monomial map
$f_A:X(\Delta)\dashrightarrow X(\Delta)$ is $k$-stable. Then
the degree sequence $\deg_{D,k}(f_A^n)$ satisfies a linear recurrence.
\end{prop}

For the proof we will need the \emph{Caley-Hamilton theorem}:
Let $B\in M_L(\Z)$ and assume that
\[
\chi (r)=r^L+\varphi_{L-1}r^{L-1}+\cdots +\varphi_1 r+ \varphi_0
\]
is the characteristic polynomial of
$B$. Then the Caley-Hamilton theorem asserts that
\[
B^L+\varphi_{L-1}B^{L-1}+\cdots +\varphi_1 B+ \varphi_0 I=0
\]
where $I$ is the identity matrix. In particular, for each $1\leq i,j\leq L$,
the entry $b_{ij}^n=:b_n$ of $B^n$ satisfies the linear recurrence
\begin{equation}\label{caley}
\chi (b_n): b_{n+L}+\varphi_{L-1}b_{n+L-1}+\cdots +\varphi_1
b_{n+1}+\varphi_0 b_n=0.
\end{equation}

\begin{proof}[Proof of Proposition~\ref{stabdeg}]
Since $D$ is ample, the class
$[D]^k$ in $H^{2k}(X)$ is non-zero, where $X=X(\Delta)$, and thus we can extend it to a basis $[D]^k,
\theta_1,\ldots, \theta_r$ for $H^{2k}(X)$ such that $\theta_j\cdot
[D]^{m-k}=0$ for $j=1,\ldots, r$. Note that then $\deg_{D,k}(f_A)$ is
equal to
$[D]^k\cdot[D]^{m-k}=[D]^m$ times
the $(1,1)$-entry of the matrix $B$ of $f_A^*:H^{2k}(X)\to
H^{2k}(X)$ with respect to the
basis $[D]^k,
\theta_1,\ldots, \theta_r$.
Since by assumption $f_A$ is $k$-stable, i.e.,
$(f_A^n)^*=(f_A^*)^n$ for all $n\in\N$, it follows that
$\deg_{D,k}(f_A^n)$ is equal to
$[D]^{m}$ times the $(1,1)$-entry of $B^n$. Therefore by the
Caley-Hamilton theorem $\deg_{D,k}(f_A^n)=:b_n$ satisfies the linear recurrence \eqref{caley}, where $\chi$ is the characteristic equation of $B$.
\end{proof}


Note that Proposition~\ref{stabdeg} implies that if $A$ satisfies
the assumption of Theorem~A, $X(\Delta)$ is the toric variety
constructed in the proof of Theorem~A, and $D$ is an ample divisor on
$X(\Delta)$, then the degree sequence $\deg_{D,k}(f_A^n)$ of
$f_A:X(\Delta)\dashrightarrow X(\Delta)$ satisfies a linear
recurrence. Similarly if $A$ and $X(\Delta)$ are as in the (proof of)
Theorem~B, then $\deg_{D,k}(f_A^{\ell n})$ satisfies a linear
recurrence for $\ell \geq \ell_0$.

Moreover, even if $f_A$ is not $k$-stable, as long as we can
lift it to a $k$-stable map, we still have a linear recurrence for its degree sequence.

\begin{prop}
\label{lift_stab_deg}
Suppose that $X$ is a simplicial projective toric variety, and that $\pi: \widetilde
X\to X$ is a nonsingular projective modification of $X$ such that
$f_A:X\dashrightarrow X$ lifts to a $k$-stable map
$f_A:\widetilde X\dashrightarrow \widetilde X$.
Then, for any ample divisor $D$ on $X$, the degree
sequence $\deg_{D,k}(f_A^n)$ satisfies a linear recurrence.
\end{prop}

\begin{proof}
Since $D$ is ample,
$\pi^*([D]^k)$ is nonzero. As in the proof of Proposition~\ref{stabdeg},
we can extend it to a basis of $H^{2k}(\widetilde X)$ in such a way that
$\deg_{D,k}(f_A)$ is equal to $\pi^*([D]^m)$ times the $(1,1)$-entry of the matrix $B$ of $f^*_A$.
Thus, again as in the proof of Proposition~\ref{stabdeg}, $\deg_{D,k}(f_A^n)$ satisfies the linear
recurrence given by the characteristic equation of $B$.
\end{proof}

It follows from Theorem~C' and Proposition~\ref{lift_stab_deg} that if $A$ satisfies the assumption of Theorem~C,
one cannot $k$-stabilize $f_A$ by blowing up $\P^m$ or any other simplicial
projective toric variety.

\subsection{Computing $\deg_{D,k} (f_A)$}

To prove Theorem C' we will express $\deg_{D,k}(f_A)$ in terms of the $k\times k$-minors of $A$.
First, for a $T$-invariant divisor $D$,
Proposition~4.1 in \cite{FW} says that $\deg_{D,k}(f_A)$ can be computed as a mixed volume:
\begin{equation}\label{favrew}
\deg_{D,k}(f_A)=m! \Vol \bigl(A P_D [k], P_D[m-k]\bigr).
\end{equation}
In the case of a general ample divisor $D'$, notice that the degrees only depend on the cohomology class of $D'$,
and there is always a $T$-invariant divisor $D$ such that $[D]=[D']$.

Next, we will describe a method of computing the mixed volume of
polytopes that we learned from Huber-Sturmfels  \cite{HS}.
A more detail exposition can be found in their paper.

Let $\mathcal P=(P_1,\ldots, P_s)$ be a sequence of polytopes in
$\R^m$ such that $P:=P_1+\cdots + P_s$ has dimension $m$. A
\emph{cell} of $\mathcal P$ is a tuple $\mathcal C = (C_1,\ldots,
C_s)$ of polytopes $C_i\subset P_i$. Let $\# C_i$ be the number of
vertices of $C_i$ and let $C:=C_1+\cdots + C_s$. A \emph{fine mixed subdivision} of $\mathcal P$ is a collection of cells $\mathcal S=\{\mathcal C^{(1)},\ldots, \mathcal C^{(r)}\}$ such that
$C^{(j)}$ has dimension $m$, $$\dim
C_1^{(j)}+\cdots +\dim C_s^{(j)}=m  ~~~~\text{ and } ~~~~
\# (C_1^{(j)})+\cdots + \# (C_s^{(j)})-s=m
$$
 for $j=1,\ldots, r$.
Moreover, $C^{(j)}\cap C^{(j')}$ is a face of both $C^{(j)}$ and $C^{(j')}$ for all $j,j'$, and $\bigcup_j C^{(j)}=P$.
If $\mathcal S$ is a fine mixed subdivsion of $\mathcal P$, then Theorem ~2.4 in \cite{HS} asserts that
\begin{equation}\label{uggla}
\Vol(P_1[k_1], \ldots, P_s[k_s])=
k_1!\cdots k_s! \cdot\sum_{\mathcal C^{(j)}\in\mathcal S, \dim
  C^{(j)}_i=k_i, i=1,\ldots, s}
\Vol(C^{(j)}).
\end{equation}
Moreover, Algorithm ~2.9 in \cite{HS} gives a method of finding fine
mixed subdivisions; in particular, each sequence of polytopes
$\mathcal P$ admits a fine mixed subdivision, where, for each $i,j$, $C^{(j)}_i$ is the convex hull of a subset of the vertices of $P_i$.

We want to apply this method to the right hand side of
\eqref{favrew}.
Assume that $P_D$ has vertices $v_1,\ldots ,v_N$. Then $A P_D$ has
vertices $A v_1,\ldots, A v_N$ and thus we can find a fine mixed
subdivision $\mathcal S$ of $\big (A P_D, P_D\big )$ with cells of
the form
\[
\mathcal C_{IJ}:=\big (\conv (A v_{i_0},\ldots, A v_{i_k}), \conv
(v_{j_0},\ldots,v_{j_{m-k}})\big ),
\]
where
$\conv(v_{i_0},\ldots,v_{i_{k}})$ denotes the convex hull of $v_{i_0},\ldots,v_{i_{k}}$, for some $I=\{i_0,\ldots, i_k\}$
and $J=\{j_0,\ldots, j_{m-k}\}\subset [N]$. Let $\mathcal S_k\subset \mathcal S$ be the set of cells $\mathcal C_{IJ}$, where $|I|=k+1$. Then \eqref{uggla} gives
\[
\Vol(A P_D[k], P_D[m-k])=
k! (m-k)!
\sum_{\mathcal C_{IJ}\in \mathcal S_k}
\Vol(C_{IJ}).
\]

Note that $C_{IJ}$ is the Minkowski sum of the $k$-simplex $\conv (A
v_{i_0},\ldots, A v_{i_k})$ with edges
$A(v_{i_1}-v_{i_0}), \ldots, A(v_{i_k}-v_{i_0})$ and the $(m-k)$-simplex $\conv
(v_{j_0},\ldots,v_{j_{m-k}})$ with edges $v_{j_1}-v_{j_0},\ldots,
v_{j_{m-k}}-v_{j_0}$.
From now on,
let us fix a basis of $M$.
It follows that
$k!(m-k)!\Vol (C_{IJ})$ equals the modulus of the determinant of the matrix $B_{IJ}$ with the vectors
$A(v_{i_1}-v_{i_0}), \ldots, A(v_{i_k}-v_{i_0})$ and $v_{j_1}-v_{j_0},\ldots, v_{j_{m-k}}-v_{j_0}$ as columns.
Since $P_D$ is a lattice polytope, the determinant of $B_{IJ}$ is an integer linear combination of
$k\times k$-minors of $A$.
Hence $\deg_{D,k}(f_A)$ is of the form
$\sum \sigma_{IJ}A_{IJ}$ where $\sigma_{IJ}\in\Z$ and $A_{IJ}$ is the $IJ$-minor of $A$.
Observe that the matrix $\sigma$ with entries $\sigma_{IJ}$ only depends on the set of multi-indices $I,J$ such that $C_{IJ}$ is in $\mathcal S_k$ and the sign of the determinant of $B_{IJ}$.
Since there are only finitely many choices of $I,J$ and signs, we conclude the following.

\begin{lma}\label{flyg}

There is a finite set $\Sigma$ of matrices $\sigma=(\sigma_{IJ})\in\Z^{{m\choose k}^2}$, such that for each $A:M\to M$ there is a $\sigma=\sigma(A)\in\Sigma$ such that
\[
\deg_{D,k}(f_A)=\sum_{IJ}\sigma_{IJ} A_{IJ}.
\]

\end{lma}

\subsection{Proof of Theorem C'}
Our proof is much inspired by the proof of Proposition~3.1 in \cite{BK} and the proof of Proposition~7.3 in \cite{HP}. We will
argue by contradiction using a result from combinatorics, which says
that if $\alpha_n$ and $\beta_n$ are two sequences that each satisfies
a linear recurrence, then the set of $n\in \N$, for which $\alpha_n=\beta_n$,
is either finite or eventually periodic, see \cite[Chapter~4, Exercise~3]{St}.
In particular if $\alpha_n=\beta_n$ for infinitely many $n$, then for some $a,b\in \N$,
\begin{equation}\label{joho}
\alpha_{a+b\ell}=\beta_{a+b\ell} \text{ for }\ell \gg 0.
\end{equation}

Now, let $\alpha_n =\deg_{D,k}(f_A^n)$.
By Lemma \ref{flyg} $\alpha_n=\sum_{IJ} \sigma_{IJ}(n)A^n_{IJ}$, where
$A_{IJ}^n$ is the $IJ$-minor of $A^n$, for
some matrix $\sigma(n)\in\Sigma$. Since $\Sigma$ is a finite set,
there is at least one $\sigma\in\Sigma$ such that $\sigma(n)=\sigma$ for
infinitely many $n$. Pick such a $\sigma=(\sigma_{IJ})\in\Sigma$ and let
$\beta_n=\sum_{IJ}\sigma_{IJ}A_{IJ}^n$.
Let $\chi (r)$ be the characteristic polynomial of
$\Lambda^kA$.
%
By the Caley-Hamilton theorem the
entries $A^n_{IJ}$ of $(\Lambda^k A)^n$, cf.\ \eqref{CB}, satisfy the
linear recurrence
$\chi (A_{IJ}^n)$, see \eqref{caley}.
It follows that $\beta_n$ satisfies the linear recurrence $\chi (\beta_n)$,
and $\alpha_n=\beta_n$ for infinitely many $n$.



Next, we claim that, if $A$ is as in the assumption, then for each choice of $a,b\in\N$, $\beta_{a+b\ell}<0$
for infinitely many $\ell$.
Since the eigenvalues of $A$ satisfy
$$|\mu_{k-1}|<|\mu_k|=|\mu_{k+1}|<|\mu_{k+2}|$$ it follows that
$\mu_{[k]}=:\mu$ and $\mu_{\{1,\ldots, k-1,k+1\}}$ are the two eigenvalues of
$\Lambda^kA$ of largest modulus, and that the other eigenvalues
$\mu_{I_3}, \ldots, \mu_{I_{{m\choose k}}}$ are of strictly smaller
modulus. Moreover, since $\mu_{k+1}=\bar\mu_{k}$ and
$\mu_k/\bar\mu_{k}$ is not a root of unity,
it follows that $\mu_{\{1,\ldots, k-1,k+1\}}=\bar\mu$ and that
$\mu/\bar\mu$ is not a root of unity.
Hence we can write
\[
(\Lambda^k A)^n=
P
\begin{bmatrix}
\mu^n & 0 & 0 & \cdots \\
0 & \bar \mu^n & 0 & \cdots \\
0 & 0 & \mu_{I_3}^n & \\
\vdots & \vdots & & \ddots
\end{bmatrix}
P^{-1}
\]
for some invertible matrix $P$.
It follows that
\[
\beta_n=\sum \sigma_{IJ}A^n_{IJ}=C\mu^n + D\bar \mu^n+\O
(|\mu_{I_3}|^n)
\]
where $C$ and $D$ are independent of $n$.
Since $\beta_n$ is real it follows that $D=\bar C$, so that
\[
\beta_n=2\Re (C) \cdot\Re (\mu^n) + \O (|\mu_{I_3}|^n).
\]
Since $\mu= |\mu|\cdot e^{2\pi i\theta}$ with $\theta\not\in\Q$, it
follows that $\arg(\mu^{a+b\ell})$ is dense in $[0,2\pi)$. In
  particular, $\Re (\mu^{a+b\ell})<0$ for infinitely many $\ell$ and $\Re (\mu^{a+b\ell})>0$ for infinitely many $\ell$, and
  thus, since $|\mu_{I_3}|<|\mu|$, $\beta_{a+b\ell}<0$ for infinitely
  many $\ell$.


Assume that $\alpha_n$ satisfies a linear recurrence. Then, since
$\alpha_n=\beta_n$ for infinitely many $n$,
 \eqref{joho} holds for some $a,b$, but this contradicts that $\alpha_n
>0$. This proves Theorem C'.

\end{document}